\documentclass[12pt]{amsart}

\usepackage{amsmath,amssymb}
\usepackage{amsfonts}
\usepackage{verbatim}
\usepackage{longtable}
\usepackage{lscape}
\usepackage[latin1]{inputenc}
\usepackage{graphicx}
\usepackage{inputenc}
\usepackage{hyperref}

\newtheorem{theo}{Theorem}[section]
\newtheorem{coro}[theo]{Corollary}
\newtheorem{lemma}[theo]{Lemma}
\newtheorem{propo}[theo]{Proposition}
\newtheorem{defn}[theo]{Definition}
\newtheorem{rk}[theo]{Remark}
\newcommand{\1}{1\hspace{-.55ex}\mbox{l}}
\newcommand{\R}{\mathbb{R}}
\numberwithin{equation}{section}
\newcommand{\osc}{\mathop{{\rm osc}}}

\begin{document}
\title[Self-interacting diffusions: convergence in distribution]{Convergence in distribution of some 
self-interacting diffusions: the simulated annealing method}
\author{S\'ebastien CHAMBEU and Aline KURTZMANN}
\thanks{A.K. is partially supported by the Swiss National Science Foundation grant
200020-112316/1}

\maketitle

\begin{abstract}
We study some self-interacting diffusions living on $\mathbb{R}^d$ solutions to: $$\mathrm{d}X_t = \mathrm{d}B_t - g(t)\nabla V(X_t - \overline{\mu}_t) \mathrm{d}t$$ where $\overline{\mu}_t$ is the empirical mean of the process $X$, $V$ is an asymptotically strictly convex potential and $g$ is a given function, not increasing too fast to the infinity or constant. The authors have already proved that the ergodic behavior of $X$ is strongly related to $g$. We go further and, using the simulated annealing method, we give some conditions for the convergence in distribution of $X$  toward $X_\infty$ (which law is related to the global minima of $V$). We also investigate the case $g(t)= 1$.
\end{abstract}

\section{Introduction}
In \cite{CK}, the authors have obtained some conditions for both the pointwise
ergodicity and the almost sure convergence of some
self-interacting diffusions. We will go further in the study of
such processes. The aim of this paper is to obtain some conditions first, for the convergence in probability, and second, the convergence in distribution of the self-interacting diffusion $X$ defined by
\begin{equation} \label{equX1}
\mathrm{d}X_t = \mathrm{d}B_t - g(t)\nabla V(X_t -
\overline{\mu}_t) \mathrm{d}t,\,\ X_0 = x
\end{equation}
where $B$ is a standard Brownian motion and $\overline{\mu}_t$
denotes the empirical mean of $X$: 
\begin{eqnarray}\label{defmu}
\overline{\mu}_t = \frac{1}{r+t}\left(r\bar{\mu} + \int_{0}^{t}
X_{s} \mathrm{d}s\right),\,\ \overline{\mu}_0 = \overline{\mu}.
\end{eqnarray}
Here $\mu $ is an initial probability measure on
$\mathbb{R}^d$, $\bar{\mu}$ denotes the mean of $\mu$ and $r>0$ is
an initial weight.

This paper deals with the well-known theory of simulated annealing,
which has been developed since the 80's. For physical systems, an important
question is to find the globally minimum
energy states of the system. Experimentally, the ground states are
reached by a procedure, called the chemical annealing. Let us explain the procedure. One first melts a
substance and then cools it slowly enough to pass 
through the freezing temperature. If the temperature decreases too fast, then the system does not end up into a ground state, but in a local (but not global) minimum. On the other hand, if the
temperature decreases too slowly, then the system approaches the
ground states very slowly. The competition between these two
effects determines the optimal speed of cooling, that is the
annealing schedule.

The study of the simulated annealing involves the
theory of (homogeneous and) non-homogeneous Markov chains and
diffusion processes, large deviation theory, spectral analysis of
operator and singular perturbation theory. Pioneer work has been done by Freidlin and Wentzell \cite{FW}. The
initial problem consists in finding the global minima of a 
function $U$. Indeed, one has to study the Markov
process $X^\varepsilon$ in $\mathbb{R}^d$ given by the
Langevin-type Markov diffusion (we emphasize that
$\varepsilon = \varepsilon(t)$)
\begin{equation}\label{edsrecuit}
\mathrm{d}X^\varepsilon_t = \varepsilon \mathrm{d}B_t - \nabla
U(X^\varepsilon) \mathrm{d}t.
\end{equation}
$X^\varepsilon$ may be considered as a perturbation of the
trajectory $X^0$ of the dynamical system
$\frac{\mathrm{d}X_t^0}{\mathrm{d}t} = -\nabla U(X_t^0)$. Let us explain briefly the model. If the temperature $\varepsilon$ is almost constant for
a sufficiently large amount of time, then the process
$X^\varepsilon$ and the fixed temperature process behave
approximatively the same at the end of that time-interval. Denote by $min$ the set of all the global minima of $U$. The
optimal annealing schedule (that is, $\varepsilon$), for the
convergence criterion $\mathbb{P}_x(X^\varepsilon_t \in min)\rightarrow 1$ as $t$ goes to $+\infty$, was first
determined by Hajek \cite{haj} for a finite state space. Later, Chiang,
Hwang and Sheu \cite{CHS} studied the convergence rate of the latter probability via the large
deviations of the transition density of $X^\varepsilon$. This rate is
actually strongly related to the spectral gap of the invariant
measure of $X^\varepsilon$.

Note, that Chiang, Wang and Sheu were one of the firsts to show the
convergence of the algorithm of the simulated annealing, in the case
$\varepsilon(t)^2 = k/\log t$ for $k$ large enough. Later, Royer
\cite{R} obtained the same result for $k>\Lambda$, where $\Lambda$
is related to the second eigenvalue
$\lambda_2^\varepsilon$ of the corresponding infinitesimal generator. 
Moreover, Hwang and Sheu \cite{H2} established (by probabilistic methods) the existence of $\Lambda
:= \underset{\varepsilon \rightarrow 0}{\lim} - \varepsilon^2 \log
\lambda_2^\varepsilon$. Finally, Holley and Stroock \cite{hs_rec}
initiated an other method and proved, in the discrete case, the
convergence of the simulated annealing algorithm via Sobolev's
inequality. They went further in their study with Kusuoka
\cite{HKS}. After that, Miclo \cite{M} proved, by using some functional
inequalities, that the free energy (that is the relative entropy
of the distribution of the process at time $t$ with respect to the
invariant probability at that time $t$) satisfies a differential
inequality, implying (under some decreasing evolution of the
temperature to zero) the convergence of the process to the global
minima of the potential. 

A natural question arises: what happens if the temperature, that
is $\varepsilon$, decreases too fast to zero? Then, the potential
can freeze in a local minimum (the ``choice" of this minimum
depends on the initial condition) and therefore the process
converges to this local minimum. We won't consider this case here.
\medskip

First, we study the process $Y_t := X_t -\bar{\mu}_t$, which satisfies the following SDE
\begin{equation}\label{equYV}
\left\{%
\begin{array}{ll}
    \mathrm{d}Y_{t} = \mathrm{d}B_{t} - g(t) \nabla V(Y_{t}) \mathrm{d}t - Y_t\frac{\mathrm{d}t}{r+t};\, Y_{0} = x - \overline{\mu};\\
    \mathrm{d}\overline{\mu}_t = Y_t \frac{\mathrm{d}t}{r+t}.
\end{array}%
\right.
\end{equation}
We will adapt the simulated annealing method to $Y$. We will prove that, depending on $g$, either the process $Y$ converges in distribution (and not in probability) towards a variable which is concentrated on the global minima of $V$ or converges in probability to a random variable, which support is $\R^d$. Suppose that $V=W+\chi$, where $W$ is strictly convex and $\chi$ is a compactly supported function. Define $\osc (\chi) := \sup\limits_x \chi(x) - \inf\limits_x \chi(x)$. Denote by $(m_i)_{1\le i\le n}$, respectively $(M_i)_{1\le i\le p}$, the local minima, respectively maxima and saddle points of $V$. We have proved in \cite{CK}, that there exist some nonnegative constants $a_i$ such that $\sum_i a_i =1$ and for all continuous bounded $f$: $\frac{1}{t}\int_0^t f(Y_s) \mathrm{d}s \xrightarrow {a.s.} \sum a_i f(m_i)$. 
Here are the main results of the present work, corresponding to two classes of functions $g$. First, we prove, for a logarithmic $g$, that $X$ converges under a condition of symmetry on the critical points of $V$:
\begin{theo}
Suppose $\infty>\underset{t\rightarrow \infty}{\lim} g(t)^{-1} \log G(t) = k>\max\{2\osc(\chi),d/4\}$. Then the process $X$ converges in distribution to $Y_\infty + \int_0^\infty Y_s \frac{\mathrm{d}s}{r+s}$ if and only if $V$ is such that $\underset{1\leq i \leq p}{\sum} a_i m_i =0$.
\end{theo}
Second, we show that, for a constant $g$, $X$ converges under a condition of symmetry on $V$:
\begin{theo}
Suppose that $\underset{t\rightarrow \infty}{\lim} g(t) = 1$. Then $\overline{\mu}_t$ converges in probability and $X_t$ converges in probability to $\overline{\mu}_\infty+ Y_\infty$, where $Y_\infty$ has the normalized distribution density  $e^{-2V(x)}/Z$, if and only if $\int x e^{-2V(x)} \mathrm{d}x = 0$. Else $X$ diverges.
\end{theo}

The paper is organized in the following way. In Section 2, we recall the notations and some results of \cite{CK}. Section 3 is devoted to the study of the process $Y$. In particular, we prove the convergence in distribution of $Y$ towards global minima thanks to the simulated annealing method. Afterwards, we deduce in Section 4 some conditions for the convergence of the self-interacting process $X$. Finally, we will study the constant case $g\equiv 1$ in Section 5.

\section{Notation, hypothesis and former results} 
We denote by $G$ the function $G(t) := \int_0^t g(s) \mathrm{d}s$ and $G^{-1}$ is its generalized inverse: $G^{-1}(t) := \inf \{u\ge 0; G(u) \ge t\}$. In the whole
following, $(\cdot,\cdot)$ denotes the Euclidian scalar product.
We denote by $\mathcal{P}(\mathbb{R}^d)$ the set of probability measures on $\mathbb{R}^d$.

In the sequel, we suppose that $V:
\mathbb{R}^d\rightarrow \mathbb{R}_+$ is such that:
\begin{enumerate}
    \item (\textit{regularity and positivity}) $V\in \mathcal{C}^2(\mathbb{R}^d)$ and $V\geq 0$;
    \item (\textit{convexity}) $V= W+\chi$ where $\chi$ is a compactly supported function and there exists $c>0$ such that $\nabla^2 W(x) \geq c Id$ and $\nabla \chi$ is Lipschitz (with the constant $\tilde{C}>0$);
    \item (\textit{growth}) there exist $a,b>0$ such that for all $x\in \mathbb{R}^d$, we have 
    \begin{equation}\label{growth}
    \Delta V(x)\leq a+bV(x) \,\, \text{and} \,\, \lim_{|x|\rightarrow \infty} \frac{|\nabla V(x)|^2}{V(x)} = \infty.
    \end{equation}
\end{enumerate}
We also assume that $V$ has a finite number of critical points.
Let $Max =\{M_{1}, M_{2},\cdots, M_{p}\}$ be the set of saddle
points and local maxima of $V$ and $Min=\{m_{1}, m_{2},\cdots,
m_{n}\} $ be the set of the local minima of $V$. We assume that
$\forall i$, $\forall \xi \in \mathbb{R}^d$, $(\nabla^2
V(m_i)\xi,\xi)>0$ and for all $M_i$, $\nabla^2 V(M_i)$ admits a negative eigenvalue.

We also assume that the application $g: \mathbb{R}_+ \rightarrow
\mathbb{R}_+$ belongs to $\mathcal{C}^1(\mathbb{R}_+)$ and, without any loss of generality, that $g(0)>0$. In the following, we will consider the cases $g(t) = k\log t$ and $g(t) =1$.
\begin{rk}
If $\underset{t\rightarrow \infty}{\lim} g(t) =\infty$, then for all $T>0$, we have that $G^{-1}(t+T) - G^{-1}(t) \underset{t\rightarrow \infty}{\longrightarrow} 0$.
\end{rk}

We have already shown that the SDE \eqref{equX1} studied admits a
unique global strong solution:
\begin{propo}(\cite{CK} proposition 3.3)
For any $x\in \mathbb{R}^d$, $\mu\in\mathcal{P}(\mathbb{R}^d)$ and
$r>0$, there exists a unique global strong solution $(X_t, t\geq
0)$ of \eqref{equX1}.
\end{propo}

\begin{theo}(theorem 5.6)\label{thm:cvge}
Suppose that $\lim g(t) = \infty$ and $\lim g'(t)/g^2(t) = 0$. Then, a.s. the normalized occupation measure of $Y$ converges weakly to a convex combination of Dirac measures taken in the critical points of $V$: there exist $a_i\geq 0$ such that $\sum_i a_i =1$ and for all continuous bounded function $f$, we have a.s. $$\frac{1}{t}\int_0^t f(Y_s) \mathrm{d}s \underset{t\rightarrow \infty}{\longrightarrow} \sum_{i=1}^n a_i f(m_i).$$
\end{theo}

\section{Asymptotic behavior of $Y$}\label{s:asympY}
Consider the time-changed process $Z_t := Y_{G^{-1}(t)}$, satisfying the following SDE
\begin{equation}\label{recequZ}
\mathrm{d}Z_t = \frac{1}{\sqrt{g\circ G^{-1}(t)}} \mathrm{d}W_t -
\left(\nabla V(Z_t) + \frac{Z_t}{(r+G^{-1}(t))g\circ
G^{-1}(t)}\right) \mathrm{d}t, \, Z_0=z,
\end{equation}
where $W$ is a Brownian motion such that $\frac{1}{\sqrt{g\circ G^{-1}(t)}} \mathrm{d}W_t$ has the same law as $B_{G^{-1}(t)}$\footnote{The Wiener processes $B_t$ and $W_t$ are not the same, but this does not matter because we are only interested in the probability distribution.}. We identify $1/\sqrt{g\circ G^{-1}(t)}$ as the temperature in the simulated annealing model. So, define $\varepsilon^2(t) := \frac{1}{g\circ G^{-1}(t)}$. The process $Z$ reads
\begin{equation}
\mathrm{d}Z_t = \varepsilon(t) \mathrm{d}B_t - \left(\nabla V(Z_t) + \frac{Z_t \varepsilon(t)^2}{r+G^{-1}(t)}\right) \mathrm{d}t.
\end{equation}
In this part, we suppose that $\frac{\log G(t)}{g(t)}$ is bounded and that $g'(t)/g(t)^2$ converges to zero.

\subsection{Tightness}
We begin to prove that the law of the process $Z$ is a tight family of measures.

\begin{lemma}(Chiang-Hwang-Sheu)\label{tight}
There exist some $R,c>0$ such that $$\mathbb{E}\left(V(Z_t)\1_{ \{V(Z_t)\geq
R\}}\right) \leq \frac{c}{g\circ G^{-1}(t)}.$$
\end{lemma}
\begin{proof}
Let us first exhibit a constant $\gamma$ such that $\mathbb{E}\left(V(Z_t)\right) \leq \mathbb{E} V(Z_0) + \gamma t$. By It\^o's formula, we have
\begin{eqnarray*}
\frac{\mathrm{d}}{\mathrm{d}t}\mathbb{E}V(Z_t) = -\mathbb{E}\left[\left( \nabla V(Z_t),\frac{Z_t\varepsilon(t)^2}{(r+G^{-1}(t))}\right)+\left|\nabla V(Z_t)\right|^2 - \frac{\Delta V(Z_t)}{2}\varepsilon(t)^2 \right]. 
 \end{eqnarray*}
As $\left|\nabla V\right|^2-\Delta V$ is bounded from below, the first assertion follows.
Now, we adapt the proof given by Duflo \cite{D}. The growth condition \eqref{growth} implies that there exists $r_0$ such that,
for $V(x)\geq r_0$, we get $|\nabla V(x)|^2 \geq 2cV(x)$ where $c>0$
is a constant. Let $\phi$ be a nonnegative and nondecreasing function of class $\mathcal{C}^2$, $\phi: \mathbb{R}\rightarrow
[0,1]$ such that $\phi(x) =0$ for $x\leq r_0$, $\phi(x) =1$ for $x\geq R$ where $r_0<R<\infty$. Remark, that the continuous function $\nabla (\phi\circ V) = (\phi'\circ V) \nabla V$ is bounded and consider the application $\psi := (\phi\circ V) V$. We apply It\^o's formula to the function $x\mapsto \psi(x)$ and we get
\begin{eqnarray*}
\mathrm{d} \psi(Z_t) &=& \varepsilon(t) (\nabla \psi(Z_t), \mathrm{d}W_t)+ \frac{\varepsilon(t)^2}{2} \Delta \psi(Z_t) \mathrm{d}t\\
&-& \left(\frac{\psi(Z_t)}{ V(Z_t)} + V(Z_t) \phi' \circ V(Z_t)\right)\left( |\nabla V (Z_t)|^2 + \frac{\varepsilon(t)^2}{r+ G^{-1}(t)} (Z_t, \nabla V(Z_t))\right) \mathrm{d}t.
\end{eqnarray*}
Let $\alpha (t) := \mathbb{E}[\phi\circ V(Z_t) V(Z_t)]$. By the first assertion, $\alpha$ is well defined.
Due to 
the compact support of $\phi'$ and 
because we can decompose $\Delta \psi$ as
\begin{eqnarray*}
\Delta \psi(z) &=& \phi''\circ V(z) V(z) |\nabla V(z)|^2 + \phi'\circ V(z) V(z) \Delta V(z)\\ 
&+& 2 \phi'\circ V(z) |\nabla V(z)|^2 + \phi \circ V(z) \Delta V(z),
\end{eqnarray*}
there exists $\widetilde{C}>0$ such that
$\Delta \psi \leq \widetilde{C} ((\phi \circ V)V +1)$. So, we get the bound 
\begin{eqnarray*}
\int_t^{t+h} \varepsilon(s)^2 \mathbb{E}\left(\Delta \psi(Z_s)\right)\mathrm{d}s &\leq & \widetilde{C}\left(G^{-1}(t+h)-G^{-1}(t)\right)
+ \widetilde{C} \int_{t}^{t+h} \alpha(s)\varepsilon(s)^2 \mathrm{d}s.
\end{eqnarray*}
On the other hand, we have the lower bound
\begin{eqnarray*}
&\mathbb{E}& \left[ \left( \phi\circ V(Z_t) + V(Z_t) \phi'\circ V(Z_t)\right) |\nabla V (Z_t)|^2 \right]\\ 
&\geq &  2c \mathbb{E} \left[ \left( \phi\circ V(Z_t) + V(Z_t) \phi'\circ V(Z_t)\right) V (Z_t) \right]\geq  2c \alpha(t).
\end{eqnarray*}
For $r\ge r_0$ large enough, if $V(z)\geq r$ then $(z,\nabla V(z)) \geq 0$, so we get 
\begin{eqnarray*}
\mathbb{E} \left( \left( \phi\circ V(Z_t) + V(Z_t) \phi'\circ V(Z_t)\right) (Z_t, \nabla V(Z_t)) \right)  \geq  0.
\end{eqnarray*}
Therefore, the preceding It\^o's formula leads to 
\begin{eqnarray*}
\alpha(t+h) - \alpha (t) &\leq & -2c\int_t^{t+h}\alpha (s) \mathrm{d}s + \frac{\widetilde{C}}{2}\int_t^{t+h} \alpha(s)\varepsilon(s)^2 \mathrm{d}s\\ 
&+& \frac{\widetilde{C}}{2}(G^{-1}(t+h) - G^{-1}(t)),
\end{eqnarray*}
Letting $h$ go to zero, this yields to $\alpha'(t) \leq -2c \alpha(t) + \frac{\widetilde{C}}{2} (\alpha(t)+1) \varepsilon(t)^2$. Choose $r$ large enough, so that we have $$\alpha'(t) \leq -c \alpha(t) + \frac{\widetilde{C}}{2}\varepsilon(t)^2.$$
In order to solve this inequation, let $\alpha(t) := \beta(t) e^{-ct}$. We have $\beta(t)\leq \int_0^t \frac{\widetilde{C}e^{cs}}{2}\varepsilon(s)^2\mathrm{d}s$. As $\lim_{t\to + \infty} \frac{g'(t)}{g^2(t)} = 0$, this yields to $\alpha (t)\leq C_1\varepsilon(t)^2/c$ where $C_1$ is a positive constant independent of $c$. To conclude, we just need to remark that $\1_{\{V\geq R\}} \leq \phi \circ V$.
\end{proof}

\begin{coro}\label{tension}
The time marginal laws of the processes $Y$ and $Z$ are tight.
\end{coro}
\begin{proof}
The previous result implies $$\mathbb{E}V(Z_t) = \mathbb{E}(V(Z_t)\1_{V(Z_t) < R}) + \mathbb{E}(V(Z_t)\1_{V(Z_t) \geq R}) \leq R + \frac{c}{g(0)}.$$ So, for a constant $A>0$, there exists a compact set $K$ such that $\{V\leq A\} \subset K$ and $$\mathbb{P}(Z_t \in K) \geq \mathbb{P}(V(Z_t) \leq A) \geq 1-\frac{\mathbb{E}(V(Z_t))}{A} \underset{A \to \infty}{\rightarrow} 1.\qedhere$$
\end{proof}

\subsection{Convergence in distribution towards the global minima of $V$}
Remember that $\varepsilon^2(t) := \frac{1}{g\circ G^{-1}(t)}$ and let
$a(t) := (r+G^{-1}(t)) \varepsilon^{-2}(t)$. The process $Z$ reads
\begin{equation}\label{recuit}
\mathrm{d}Z_t = \varepsilon(t) \mathrm{d}B_t - \nabla V_t(Z_t) \mathrm{d}t
\end{equation}
where we have defined $V_t(x) := V(x) + \frac{|x|^2}{a(t)}$.
Actually, we will prove that this non-homogeneous Markov process
converges in distribution to its ``invariant" probability measure.
Of course, if we suppose that $a(t) \equiv a$ and $\varepsilon(t)\equiv \varepsilon$, then the convergence in distribution is well-known. 

Let $L_{t,\varepsilon}$ be the operator defined by
$L_{t,\varepsilon} := \frac{1}{2}\varepsilon^2\Delta - (\nabla
V_t, \nabla )$ and $P^{t,\varepsilon}$ the associated semigroup. Define 
\begin{equation}\label{probainv}
\Pi_{t,\varepsilon}(\mathrm{d}x) := \frac{1}{Z(t,\varepsilon)}
e^{-2\varepsilon^{-2} V_t(x)} \mathrm{d}x,
\end{equation}
where $Z(t,\varepsilon) := \int_{\mathbb{R}^d} e^{-2\varepsilon^{-2} V_t(x)} \mathrm{d}x$. As $|\nabla V_t|^2 - \Delta V_t$ is bounded from below, the theory of Schr\"odinger operator implies that
$L_{t,\varepsilon}$ is self-adjoint in $L^2(\Pi_{t,\varepsilon})$ and admits a
spectral gap. Furthermore, when $|\nabla V_t|^2 - \Delta V_t$ goes
to the infinity as $|x|\rightarrow \infty$ then the spectrum of
$L_{t,\varepsilon}$ is discrete: $0=\lambda_1(t,\varepsilon)<-\lambda_2(t,\varepsilon)< \ldots$. Heuristically, when the time is of order $e^{\varepsilon^{-2} \lambda_2}$, the
transition density has a nice lower bound and the process is very
close to the ``invariant probability'' $\Pi_{t,\varepsilon}$. So, our main goal is to compute the convergence of the latter probability measure when $t$ goes to the infinity. What is more, as the subspace corresponding to the first eigenvalue
$\lambda_1(t,\varepsilon)$ is composed by the constant functions, we find
 $$\lambda_2(t,\varepsilon) =\inf\left\{\int |\nabla
\phi|^2 \mathrm{d}\Pi_{t,\varepsilon}; \quad
\mathrm{Var}_{\Pi_{t,\varepsilon}} (\phi) =1, \phi \in
\mathcal{C}^\infty(\mathbb{R}^d)\right\}.$$ So, our first aim is to compute
the eigenvalue $\lambda_2$ and study its behavior when
$t\rightarrow \infty$ (that is $\varepsilon\rightarrow 0$).  

Consider for a while $\Pi_{\infty,\varepsilon}:= \frac{1}{Z(\varepsilon)}
e^{-2\varepsilon^{-2} V(x)} \mathrm{d}x$,
where $Z(\varepsilon) := \int_{\mathbb{R}^d} e^{-2\varepsilon^{-2} V(x)} \mathrm{d}x$. Let $min=\{m_1,\ldots, m_q\}$ be the set of the global minima of $V$. Hwang \cite{H} has established that $\Pi_{\infty,\varepsilon}$ converges weakly when $\varepsilon$ converges to zero and described the limit:
\begin{lemma}(\cite{H})
When $t$ goes to the infinity, the probability measure $\Pi_{\infty,\varepsilon(t)}$ converges weakly
to 
\begin{equation}\label{eq:pi0}
\Pi_0 :=\frac{1}{\underset{1\leq i\leq q}{\sum} (\mathrm{det} \nabla^2 V(m_i))^{-1/2}} \underset{1\leq i\leq q}{\sum} (\mathrm{det} \nabla^2 V(m_i))^{-1/2}\delta_{m_i}.
\end{equation} 
\end{lemma}
\begin{rk}
We use here a weaker form than the result of Hwang, who has actually proved the convergence of $\Pi_{\infty,\varepsilon(t)}$ for a more general set $min$.
\end{rk}

We can now state and prove the
\begin{propo}\label{convrecuit}
Suppose that (asymptotically) $g(t)> \frac{d}{4}\log t$. The probability measure $\Pi_{t,\varepsilon(t)}$ converges weakly
to $\Pi_0$ as $t$ goes to the infinity, where $\Pi_0$ is defined by \eqref{eq:pi0}. 
\end{propo}
\begin{proof}
Recall, that $\varepsilon^2(t)a(t)=2(r+G^{-1}(t))$ and 
\begin{eqnarray*}
Z(t,\varepsilon(t)) &=& \int_{\mathbb{R}^d} e^{-2\varepsilon^{-2}(t)V(x)}e^{-2\frac{|x|^2}{\varepsilon^{2}(t)a(t)}}\mathrm{d}x.
\end{eqnarray*}
Let $K$ be the compact set $K:=\{x | V(x)\leq 1\}$. There exists a constant $A>0$ such that $K\subset B(0,A)$. Then, on one hand, we get the upper bound
\begin{eqnarray*}
\int_{K^c} e^{-2\varepsilon^{-2}(t)V(x)}e^{-2\frac{|x|^2}{\varepsilon^2(t)a(t)}}\mathrm{d}x \leq  \int_{\mathbb{R}^d} e^{-2\varepsilon^{-2}(t)}e^{-2\frac{|x|^2}{\varepsilon^2(t)a(t)}}\mathrm{d}x = \frac{(2\pi a(t)\varepsilon^2(t))^{d/2}}{e^{2\varepsilon^{-2}(t)}}.
\end{eqnarray*}
On the other hand, we obtain similarly the lower bound
\begin{eqnarray*}
\int_{K} e^{-2\varepsilon^{-2}(t)V(x)}\mathrm{d}x \geq \int_{K} e^{-2\varepsilon^{-2}(t)V(x)}e^{-2\frac{A}{a(t)\varepsilon^2(t)}}\mathrm{d}x.
\end{eqnarray*}
By Laplace's method, we have (see \cite{H,Z})
$$\int_{K} e^{-2\varepsilon^{-2}(t)V(x)}\mathrm{d}x \, \, \underset{t\to+\infty}{\sim} \; \; \sum_{i=1}^q (\pi\varepsilon^2(t))^{d/2} (\det \nabla^2V(m_i))^{-1/2},$$
where $(m_i)_i$ are the global minima of $V$ (they form a finite set). On the other hand, as $g(t)>\frac{d}{4} \log t$, we have that $G^{-1}(t)$ goes to the infinity and so $$a(t)^{d/2}e^{-2\varepsilon^{-2}(t)} = (G^{-1}(t) g\circ G^{-1}(t))^{d/2} e^{-2g\circ G^{-1}(t)}\underset{t\to \infty}{\longrightarrow} 0.$$ As a consequence, we find the asymptotic equivalence
$$Z(t,\varepsilon(t)) \; \; \underset{t\to+\infty}{\sim} \;\; \sum_i (\pi\varepsilon^2(t))^{d/2} (\det \nabla^2V(m_i))^{-1/2}.$$
By the same method, if $\phi$ is a continuous function, with compact support containing only the global minimum $m_1$, we find 
$$ \int_{R^{d}} \phi(x) e^{-2\varepsilon^{-2}(t)V(x)}e^{-2\frac{|x|^2}{a(t)\varepsilon^2(t)}}\mathrm{d}x \underset{t\to+\infty}{\sim} (\pi\varepsilon^2(t))^{d/2} (\det \nabla ^2V(m_1))^{-1/2}\phi(m_1).$$
This concludes the proof and also give the explicit form of $\Pi_0$.
\end{proof}

To prove that $Z$ converges in distribution towards the global minima of $V$, 
 we follow the approach initiated by Holley and Stroock \cite{hs_rec}, Holley, Kusuoka and Stroock \cite{HKS} and Miclo \cite{M}, using some functional inequalities. We suppose in the
following that $g\circ G^{-1}(t)$ is asymptotically equivalent to $k\log(1+t)$ for $k$ large enough. The remainder of the Section is the following. First, we will show that the measures $(\Pi_{t,\varepsilon(t)},t\ge 0)$ satisfy a logarithmic Sobolev inequality. This will prove useful for the convergence of the free energy to zero. After that, we show that $Z$  converges in distribution to a random variable of law $\Pi_0$.

\begin{defn}
The measure $\mu$ satisfies the logarithmic Sobolev inequality,
with the constant $C$, denoted $LSI(C)$, if for all function $h\in
L^2(\mu)$, we have $$\int h^2\log h^2 \mathrm{d}\mu - \left(\int
h^2 \mathrm{d}\mu\right) \log \left(\int h^2 \mathrm{d}\mu\right)
\leq C \int |\nabla h|^2 \mathrm{d}\mu.$$
\end{defn}

Let $p(s,x,t,y)$ denote the density of the semi-group corresponding to the non-homogeneous Markov process $Z$. We remind that we supposed $V = W+\chi$ and $c$ is the convexity constant of $W$.

\begin{lemma}\label{logSob}
The family of probability measures $(\Pi_{t,\varepsilon(t)}, t\geq 0)$ satisfies a logarithmic Sobolev inequality $LSI(C(t))$, where $C(t) = 2e^{\frac{2}{\varepsilon(t)^2} \osc (\chi)}/c$ (and $\osc(\chi) =
\sup\chi - \inf \chi$).
\end{lemma}
\begin{proof}
The point is to use the celebrated Bakry-Emery $\Gamma_2$-criterion
(see \cite{BakEm}). Indeed, to the operator
$L_{t,\varepsilon(t)}$, we can associate the operator ``carr\'e du
champ": for all function $f,g\in \mathcal{C}^\infty$
\begin{equation}
\Gamma_t^V(f,g) := \frac{1}{2} \left(L_{t,\varepsilon(t)}(fg) -
fL_{t,\varepsilon(t)}g - gL_{t,\varepsilon(t)}f \right).
\end{equation}
Then, we define the operator $\Gamma_2^V$ as
\begin{equation}
\Gamma_2^V(t)(f) := \frac{1}{2} \left(L_{t,\varepsilon(t)}
\Gamma_t^V(f,f) - 2 \Gamma_t^V( f,L_{t,\varepsilon(t)}f) \right).
\end{equation}
The $\Gamma_2$-criterion asserts that if there exists a positive constant $C$ such that $\Gamma_2^{V_t} \geq C \Gamma_t^{V_t}$, then $\Pi_{t,\varepsilon(t)}$ satisfies a logarithmic Sobolev inequality, $LSI(2/C)$. 
An easy calculation, for all function $f$ of class
$\mathcal{C}^\infty$, leads to $$\Gamma_t^V(f,f) = \frac{\varepsilon(t)^2}{2}|\nabla f|^2$$ and $$\Gamma_2^V(t)(f) =
\frac{\varepsilon(t)^2}{2}(\nabla f, \nabla^2 V \nabla f)+ \frac{\varepsilon(t)^4}{4} ||\nabla^2 f||^2 +
\frac{\varepsilon(t)^2}{2 a(t)} |\nabla f|^2.$$ 
Recall, the decomposition $V=W +\chi$ where $W$ is
strictly convex with a constant $c$ and $\chi$ is a compactly
supported function. We apply Bakry and Emery's criterion to
the function $W + |x|^2/a(t)$ and we get that $\Gamma_2^W(t)(f) \geq c
\Gamma_t^W(f)$. So, the probability measure
$e^{-2\varepsilon^{-2}(t) (W(x) + |x|^2/a(t))} /Z$ satisfies the
inequality $LSI(2/c)$. We conclude, by Holley and Stroock's perturbation lemma \cite{HS}, that the measure
$\Pi_{t,\varepsilon(t)}$ satisfies a $LSI(C(t)$ inequality with $C(t) \le 2
e^{\frac{2}{\varepsilon(t)^2} \osc \chi} /c$.
\end{proof}
Let us denote by $(P_s^{t,\varepsilon},s\geq 0)$ the $\mathcal{C}^0$ semigroup corresponding to the non-homogeneous Markov process $Z$. For any function $f$ and any probability measure $\mu$ (such that $f$ is $\mu$-integrable), we note $<f>_{\mu} := \int_{\mathbb{R}^d} f \mathrm{d}\mu$.
\begin{lemma}\label{spectralgap}
The probability measure $\Pi_{t,\varepsilon}$ admits a spectral gap:
there exists a constant $\lambda_2(\varepsilon)>0$ such that for
all $s\geq 0$, all continuous $f\in L^2(\Pi_{t,\varepsilon})$
$$||P_s^{t,\varepsilon} f - \Pi_{t,\varepsilon} f||_{L^2(\Pi_{t,\varepsilon})}
\leq e^{-\lambda_2(t,\varepsilon) s}
\mathrm{Var}_{\Pi_{t,\varepsilon}}(f).$$
\end{lemma}
\begin{proof}
As $\Pi_{t,\varepsilon(t)}$ satisfies the inequality $LSI(C(t))$, we get the spectral inequality
with constant $\lambda_2(t,\varepsilon(t))>0$.
\end{proof}

We want to use the previous functional inequalities in order to prove the convergence of $Z_t$ (and thus $Y_t$) towards the global minima of $V$.
\begin{defn}
The free energy (up to an additive constant), or relative Kullback information, of a measure $P$ absolutely continuous with respect to $\Pi$ is: $H(P|\Pi) := \int \mathrm{d}P\log\frac{\mathrm{d}P}{\mathrm{d}\Pi}.$ 
Equivalently, if we suppose that $P$ (respectively $\Pi$) has a density $p$ (respectively $\pi$) with respect to the Lebesgue measure $\lambda$, then we define
\begin{equation}
H(p|\Pi) := \int p \log\frac{p}{\pi} \mathrm{d}\lambda.
\end{equation}
\end{defn}

\begin{propo}\label{ineq}
For all initial $t_0, x_0$, we get
\begin{eqnarray*}
\frac{\mathrm{d}}{\mathrm{d}t}H\left(p(t_0,x_0,t,\cdot)|\Pi_{t,\varepsilon(t)}\right)
&\leq & -\frac{2}{C(t)} \varepsilon(t)^2
H\left(p(t_0,x_0,t,\cdot)|\Pi_{t,\varepsilon(t)}\right)\\
&-& 4 \dot{\varepsilon}(t) \varepsilon(t)^{-3}
\int p(t_0,x_0,t,\cdot)(V_t -<V_t>_{\Pi_{t,\varepsilon(t)}}) \mathrm{d}\lambda\\
&+& \frac{2}{\varepsilon(t)^2} \int p(t_0,x_0,t,\cdot) (\dot{V}_t
-<\dot{V}_t>_{\Pi_{t,\varepsilon(t)}}) \mathrm{d}\lambda.
\end{eqnarray*}
\end{propo}
\begin{proof}
We adapt the proof of Holley and Stroock \cite{hs_rec} (and Miclo \cite{M}). In order to shorten notation, let $p_t := p(t_0,x_0,t,\cdot)$ be the distribution law of the process $Z_t$, knowing that $Z_{t_0} = x_0$ (the existence of a density $p_t$ follows from Girsanov theorem). Recall, that the family of probability measures $(\Pi_{t,\varepsilon(t)},t\geq 0)$ satisfies a family of Sobolev logarithmic inequalities $(LSI(C(t)), t\ge 0)$ and $\Pi_{t,\varepsilon(t)}( \mathrm{d}x) = \pi_{t,\varepsilon(t)}(x) \lambda(\mathrm{d}x)$. Let us introduce $$h_t := \sqrt{\frac{p_t}{\pi_{t,\varepsilon(t)}}},$$ satisfying in particular $\int h_t^2 \mathrm{d}\Pi_{t,\varepsilon(t)} = 1$. So, by Lemma \ref{logSob}, there exists a constant $C(t)$ such that
\begin{equation}\label{LSI}
H(p_t| \Pi_{t,\varepsilon(t)}) = \int p_t \log
\frac{p_t}{\pi_{t,\varepsilon(t)}} \mathrm{d}\lambda \leq C(t)
\int |\nabla h_t|^2 \mathrm{d}\Pi_{t,\varepsilon(t)}.
\end{equation}
Computing the gradient function of $h_t$, we get
$$\nabla h_t = \frac{\sqrt{p_t}}{2\pi_{t,\varepsilon(t)}} \left( \frac{\nabla p_t}{p_t} + 2 \frac{\nabla V_t}{\varepsilon(t)^2}\right).$$ We put this last estimate in the preceding inequality to get
\begin{equation}\label{eq:H}
H(p_t|\Pi_{t,\varepsilon(t)}) \leq \frac{C(t)}{4} \int p_t \left| \frac{\nabla p_t}{p_t} + 2\frac{\nabla V_t}{\varepsilon(t)^2} \right|^2 \mathrm{d}\lambda. 
\end{equation}
Moreover, the time-derivative of the free energy $H$ is
\begin{equation}\label{deriv}
\frac{\mathrm{d}}{\mathrm{d}t} H(p_t|\Pi_{t,\varepsilon(t)}) = \int \dot{p}_t \log \frac{p_t}{\pi_{t,\varepsilon(t)}}
\mathrm{d}\lambda - \int p_t \frac{\dot{\pi}_{t,\varepsilon(t)}}{\pi_{t,\varepsilon(t)}} \mathrm{d}\lambda.
\end{equation}
Our strategy is to find a upper bound for both terms of \eqref{deriv}. Kolmogorov forward equation reads
\begin{equation}\label{kolmo}
\dot{p}_t = \frac{1}{2} \varepsilon(t)^2 \Delta p_t + (\nabla p_t,
\nabla V_t) +\, \text{div}(V_t)p_t = \nabla \cdot \left(\frac{1}{2} \varepsilon(t)^2
\nabla p_t + p_t \nabla V_t\right).
\end{equation}
We have the following estimates on $\pi_{t,\varepsilon(t)}$:
\begin{equation} \frac{\dot{\pi}_{t,\varepsilon(t)}}{\pi_{t,\varepsilon(t)}} = 4\frac{\dot{\varepsilon}(t)}{\varepsilon(t)^3} \left(V_t -<V_t>_{\pi_{t,\varepsilon(t)}} \right) - \frac{2}{\varepsilon(t)^2} (\dot{V}_t - <\dot{V}_t>_{\Pi_{t,\varepsilon(t)}})\label{pi_point},
\end{equation}
\begin{equation}
\frac{\nabla \pi_{t,\varepsilon(t)}}{\pi_{t,\varepsilon(t)}} = -2\frac{\nabla V_t}{\varepsilon(t)^2} \label{nabla_pi}.
\end{equation}
Putting all the pieces together, 
integrating by parts and using \eqref{eq:H}, we get
\begin{eqnarray*}
\int \dot{p}_t \log\frac{p_t}{\pi_{t,\varepsilon(t)}} \mathrm{d}\lambda &=& \int \log \frac{p_t}{\pi_{t,\varepsilon(t)}} \nabla \cdot \left(\frac{1}{2} \varepsilon(t)^2 \nabla p_t + p_t
\nabla V_t\right) \mathrm{d}\lambda\\
&=& -\int \left(\frac{\nabla p_t}{p_t} - \frac{\nabla \pi_{t,\varepsilon(t)}}{\pi_{t,\varepsilon(t)}}, \frac{1}{2} \varepsilon(t)^2 \nabla p_t + p_t \nabla V_t\right) \mathrm{d}\lambda\\
&=& - \frac{\varepsilon(t)^2}{2} \int p_t \left|\frac{\nabla p_t}{p_t}+ 2 \frac{\nabla V_t}{\varepsilon(t)^2}\right|^2 \mathrm{d}\lambda\\
&\leq & -\frac{2}{C(t)} \varepsilon(t)^2 H(p_t |\Pi_{t,\varepsilon(t)}).
\end{eqnarray*}
On the other hand, we obtain for the second term of \eqref{deriv}:
\begin{eqnarray*}
\int p_t\frac{\dot{\pi}_{t,\varepsilon(t)}}{\pi_{t,\varepsilon(t)}} \mathrm{d}\lambda = 4
\frac{\dot{\varepsilon}(t)}{\varepsilon(t)^3} \int p_t (V_t - <V_t>_{\Pi_{t,\varepsilon(t)}}) \mathrm{d}\lambda 
- \frac{2}{\varepsilon(t)^2} \int p_t (\dot{V}_t - <\dot{V_t}>_{\Pi_{t,\varepsilon(t)}}) \mathrm{d}\lambda.
\end{eqnarray*}
Putting all the pieces together in \eqref{deriv} leads to the result.
\end{proof}

\begin{lemma}\label{majoration}
For all $t\geq 0$, the quantity $<|x|^2>_{\Pi_{t,\varepsilon(t)}}$ is bounded.
\end{lemma}
\begin{proof}
Let $K$ be the compact set $K:=\{x | V(x)\leq \eta\}$ where $\eta$
is a given positive constant. As $\Pi_{t,\varepsilon(t)}$
converges weakly to $\Pi_0$ we only need to prove that
$<|x|^2\1_{K^c}>_{\Pi_{t,\varepsilon(t)}}$ is bounded. 
For any $1\geq \varepsilon(t)^2$, we have
\begin{eqnarray*}
\int_{K^c} |x|^2 e^{-2\varepsilon^{-2}(t)V_t(x)} 
\mathrm{d}x
&\leq& \int_{K^c} |x|^2 e^{-2V(x)} e^{-2V(x)(\varepsilon^{-2}(t)-1)}\mathrm{d}x\\
&\leq&  e^{-2\eta(\varepsilon^{-2}(t)-1)} \int_{K^c} |x|^2 e^{-2V(x)}\mathrm{d}x
\end{eqnarray*}
But, as proved in Proposition \ref{convrecuit}, we have the asymptotic equivalence $$Z(t,\varepsilon(t)) \; \;
\underset{t\infty}{\sim}\;\; \sum_i (\pi\varepsilon^2(t))^{d/2}
(\mathrm{det} \nabla^2 V(m_i))^{-1/2},$$ which implies 
$ <|x|^2\1_{K^c}>_{\Pi_{t,\varepsilon(t)}} \leq \widetilde{C} \varepsilon^{-d}(t) e^{-2\eta\varepsilon^{-2}(t)} \underset{t\rightarrow\infty}{\rightarrow}0.$
\end{proof}

In order to conclude the convergence of the free energy to zero, we will use an easy result of Miclo \cite{M}
\begin{lemma}\label{gronwall}(Miclo, lemma 6)
Let $f: [0,\infty[ \rightarrow \mathbb{R}_+$ be a continuous
function such that a.s. $$f'(t) \leq \alpha_t -\beta_t f(t),$$
where $\alpha$ and $\beta$ are two continuous non-negative
functions such that $\int^\infty \beta_t \mathrm{d}t = \infty$ and
$\underset{t\rightarrow \infty}{\lim} \alpha_t / \beta_t = 0$.
Then $\underset{t\rightarrow \infty}{\lim} f(t) = 0$.
\end{lemma}
\begin{proof}
One has to adapt Gronwall's lemma. Let $g(t) =
f(t) \exp{\left(\int_0^t \beta_s \mathrm{d}s\right)}$, which is continuous. We get a.s. 
$$g'(t) \leq \alpha_t \exp{\left(\int_0^t \beta_s
\mathrm{d}s\right)}.$$ We therefore obtain that $$g(t) \leq g(0) +
\int_0^t \alpha_s \exp{\left(\int_0^s \beta_u \mathrm{d}u\right)}
\mathrm{d}s.$$ Consequently, for all $t_0\geq 0$, we have
\begin{eqnarray*}
f(t) \leq f(0) e^{-\int_0^t \beta_u \mathrm{d}u} + e^{-\int_0^t \beta_u \mathrm{d}u} \left\{\int_0^{t_0} \alpha_s e^{\int_0^s \beta_u \mathrm{d}u} \, \mathrm{d}s + \int_{t_0}^t \alpha_s e^{\int_0^s \beta_u \mathrm{d}u}\, \mathrm{d}s\right\}.
\end{eqnarray*}
Let $\eta>0$. We choose $t_0$ such that for all $t\geq t_0$, we have $\alpha_s \leq \eta \beta_s$. We thus find a upper bound, for any $t\geq t_0$, for the last term of the preceding inequality:
\begin{eqnarray*}
e^{-\int_0^t \beta_u \mathrm{d}u} \int_{t_0}^t \eta \beta_s e^{\int_0^s \beta_u \mathrm{d}u}\mathrm{d}s
\leq e^{-\int_0^t \beta_u \mathrm{d}u}\eta \left(e^{\int_0^t \beta_u \mathrm{d}u} - e^{\int_0^{t_0} \beta_u \mathrm{d}u}\right) \leq \eta.
\end{eqnarray*}
As the two first terms go to zero when $t$ goes to the infinity, we get $\underset{t\rightarrow \infty}{\limsup}f(t) \leq \eta$. 
\end{proof}

\begin{coro}\label{gronwall2}
If $\underset{t\to+\infty}{\lim} \frac{\alpha'_t}{\alpha_t\beta_t}-\frac{\beta'_t}{\beta_t^2} = 0$ and $\underset{t\to+\infty}{\lim} \frac{\alpha_t}{\beta_t}e^{\int_0^t \beta_s \mathrm{d}s} = + \infty,$ then we have asymptotically (when $t\rightarrow \infty$) $$\int_{0}^t \alpha_s e^{\int_0^s \beta_u\mathrm{d}u}\mathrm{d}s \sim \frac{\alpha_t}{\beta_t}e^{\int_0^t\beta_s\mathrm{d}s}.$$
\end{coro}
\begin{proof}
Integrating by parts, we have $$\int_{0}^t \alpha_s e^{\int_0^s \beta_u\mathrm{d}u}\mathrm{d}s =  \frac{\alpha_t}{\beta_t}e^{\int_0^t\beta_s\mathrm{d}s} - \frac{\alpha_0}{\beta_0} - \int_0^t \left( \frac{\alpha'_s}{\alpha_s\beta_s} - \frac{\beta'_s}{\beta_s^2}\right) \alpha_s e^{\int_0^s\beta_u\mathrm{d}u}\mathrm{d}s.\qedhere$$
\end{proof}

\begin{theo}
Suppose that $\varepsilon^2(t) = k/\log(t)$, with $\infty>k>2 \osc(\chi)$. Then, for all initial $t_0, x_0$, the free energy
$H\left(p(t_0,x_0,t,\cdot)|\Pi_{t,\varepsilon(t)}\right)$
converges to 0 (as $t\rightarrow \infty$).
\end{theo}
\begin{proof}
Let $t_0\geq 0$ and $x_0\in \mathbb{R}^d$. Consider the process $Z_t$, solution to the SDE
$$\mathrm{d}Z_t = \varepsilon(t) \mathrm{d}W_t -
\left(\nabla V(Z_t) +\frac{Z_t}{a(t)} \right) \mathrm{d}t, \,
Z_{t_0} = x_0.$$ The result of Proposition \ref{ineq} can be rewritten in the following way:
\begin{eqnarray*}
\frac{\mathrm{d}}{\mathrm{d}t} H(p_t|\Pi_{t,\varepsilon(t)}) &\leq
& -\frac{2}{C(t)} \varepsilon(t)^2 H(p_t|\Pi_{t,\varepsilon(t)}) +
\frac{2}{\varepsilon(t)^2} (\mathbb{E}\dot{V}_t(Z_t) -
<\dot{V}_t>_{\Pi_{t,\varepsilon(t)}})\\
&-& 4 \dot{\varepsilon}(t) \varepsilon(t)^{-3} \left(\mathbb{E}
V_t(Z_t) - <V_t>_{\Pi_{t,\varepsilon(t)}} \right).
\end{eqnarray*}
As $V(x) \geq c|x|^2$ out of a compact set and $Z$ is tight (by Corollary \ref{tension}), we have $\mathbb{E}V_t(Z_t) = O(1)$. Because $t\mapsto a(t)$ is nondecreasing while $t\mapsto \varepsilon(t)$ is nonincreasing, and as $\dot{V}_t (x) = -\frac{\dot{a}(t)}{a(t)^2} |x|^2$, the two terms $\mathbb{E}(\dot{V}_t(Z_t))$ and $<V_t>_{\Pi_{t,\varepsilon(t)}}$ are nonpositive. So, it only remains to find a upper bound for $<\dot{V}_t>_{\Pi_{t,\varepsilon(t)}}$. Indeed, by Lemma \ref{majoration}, there exist $M_1,M_2>0$ such that
$$\frac{\mathrm{d}}{\mathrm{d}t} H(p_t|\Pi_{t,\varepsilon(t)}) \leq - \frac{2}{C(t)} \varepsilon(t)^2 H(p_t|\Pi_{t,\varepsilon(t)}) + M_1 \frac{\dot{\varepsilon}(t)}{\varepsilon(t)^3} + M_2 \frac{\dot{a}(t)}{\varepsilon(t)^2 a(t)^2}.$$ 
We easily compute the time-derivative of $a(t)$:
\begin{eqnarray*}
\frac{\dot{a}(t)}{a(t)^2 \varepsilon^4(t)} &=& - \frac{2\dot{\varepsilon}(t) }{\varepsilon^3(t)(r+ G^{-1}(t))} +
\frac{1}{(g\circ G^{-1}(t))(r+G^{-1}(t))^2\varepsilon^2(t)}\\
&=& \frac{1}{kt(r+ G^{-1}(t))} +\frac{\log t}{k(g\circ G^{-1}(t))(r+G^{-1}(t))^2}.
\end{eqnarray*}
As $G^{-1}(t)$ is a nondecreasing function and because of the hypothesis on $k$, the term $\frac{C(t)}{kt(r+ G^{-1}(t))}$  converges to 0 when $t$ goes to the infinity. For the second term, we recall that $\log G(t) /g(t)$ is bounded by assumption: there exist two positive constants $m, M$ such that $mg(t) \leq \log G(t) \leq Mg(t)$. So, we get $mg(t) \leq \log (tg(t)) = \log t + \log g(t)$, what naturally implies that $g(t) = O(\log t)$ and then $G(t)\leq tg(t) = o(t^2)$. So, $$G(t)^{2\osc(\chi)/k} \log G(t) /(kg(t) (r+t)^2) \underset{t\rightarrow \infty}{\longrightarrow} 0.$$ Lemma \ref{gronwall} asserts that if $\varepsilon$ satisfies $$\int^\infty_1 \varepsilon(t)^2 \frac{\mathrm{d}t}{C(t)} = \infty \, \, \text{and} \, \, \frac{\dot{\varepsilon}(t)}{\varepsilon^5(t)} \underset{t\rightarrow \infty}{\longrightarrow} 0,$$
then $\underset{t\rightarrow \infty}{\lim} H(p_t|\Pi_{t,\varepsilon(t)}) =0$. We meet the required conditions with $\varepsilon^2(t) = k/\log t$.
\end{proof}
\begin{rk}
The constant $k$ is not optimal here, because we have used the perturbation lemma of Holley and Stroock for the estimation of the logarithmic Sobolev constant.
\end{rk}

\begin{lemma}
The speed of convergence of $H(p(t_0,x_0,t,\cdot)|\Pi_{t,\varepsilon(t)})$ toward 0 is $\frac{(\log t)^{3}}{t^{2 -2\osc(\chi) /k}}$.
\end{lemma}
\begin{proof}
As we satisfy the hypothesis of Corollary \ref{gronwall2}, the speed of convergence is $\frac{\alpha_t}{\beta_t}$, with $\alpha_t = (r+G^{-1}(t))^{-2} + (t(r+G^{-1}(t)))^{-1}$ and $\beta_t = t^{-2 \osc(\chi) /k} /\log t$. So, we get $$\frac{\alpha_t}{\beta_t} = \frac{t^{2 \osc(\chi) /k}\log t}{(r+G^{-1}(t))^2} + \frac{\log t}{t^{1-2 \osc(\chi)/k} (r+G^{-1}(t))}.$$
As, up to a multiplicative constant, $g\circ G^{-1}(t) = \log t$, we get that $g(t)=O(\log t)$ and $G^{-1}(t)$ is of the order of $t/ \log t$. So $\alpha_t /\beta_t$ is asymptotically of the order of $(\log t)^{3}t^{2 \osc(\chi) /k-2}$.
\end{proof}
\begin{rk}
It is known since the work of Freidlin and Wentzell \cite{FW} (see \cite{D}, Chap.5), that the Gibbs measure
$\Pi_{t,\varepsilon(t)}$ satisfies a large deviation principle. Therefore, the speed of convergence of $\Pi_{t,\varepsilon(t)}$ toward $\Pi_0$ is $e^{-\log t /2k}$.
\end{rk}

\begin{coro} 
Suppose that $\varepsilon^2(t) = k/\log(t)$, with $\infty>k > \max\{2\osc(\chi),d/4\}$. Then $Z$ and $Y$ converge in distribution to a random variable, which law is $\Pi_0$ \eqref{eq:pi0}.
\end{coro}
\begin{proof}
The Kullback information $H(p_t|\Pi_{t,\varepsilon(t)})$ estimates the distance between $p_t$ and $\Pi_{t,\varepsilon(t)}$ in the following way: $||p_t - \Pi_{t,\varepsilon(t)}||_{TV}^2 \leq 2 H(p_t|\Pi_{t,\varepsilon(t)})$, where $||\cdot||_{TV}$ denotes the total variation norm (see \cite{hs_rec}). The result follows because $\Pi_{t,\varepsilon(t)}$ converges weakly to $\Pi_{0}$ and the total-variation norm charaterizes the weak convergence of measures.
\end{proof}
%
\begin{rk}
We emphasize that if $\underset{t\rightarrow\infty}{\lim}g(t)^{-1} \log G(t) =k$, with $k$ not large enough, then the previous result is false and $Y$ does not converge toward the global minima of $V$ (except if each local minimum of $V$ is a global minimum). Indeed, in the simulated annealing theory, if $\varepsilon$ decreases too fast to zero, then the process $Y$ freezes in a local minimum, the choose of the minimum depending on the initial value $Y_0 = x-\overline{\mu}$.
\end{rk}

\section{The process $X$: convergence in distribution}
We give here sufficient conditions for the convergence of $X$. As usual, we begin to work with the process
$Y_t = X_t - \overline{\mu}_t$. In order to link this section with the preceding one, we recall that $\varepsilon(t)^2 = (g\circ G^{-1}(t))^{-1} = k/\log t$. So, we consider only functions $g$ such that (asymptotically) $\log G(t) =
kg(t)$. By Theorem \ref{thm:cvge}, there exist $a_i\geq 0$ such that $\sum a_i =1$ and for all continuous bounded $f$: $\frac{1}{t}\int_0^t f(Y_s) \mathrm{d}s \rightarrow \sum a_if(m_i)$. 
\begin{theo}
Suppose that $\infty> \underset{t\rightarrow \infty}{\lim} g(t)^{-1} \log G(t)= k>\max\{2\osc(\chi),d/4\}$. Then one of the
following holds:
\begin{enumerate}
    \item If $V$ is a function such that $\underset{1\leq i\leq n}{\sum} a_i m_i =0$, then $X_t$ converges in distribution to $Y_\infty + \int_0^\infty Y_s \frac{\mathrm{d}s}{r+s}$;
    \item Else, $X_t$ diverges.
\end{enumerate}
\end{theo}
\begin{proof}
We wonder whether $\int_0^t Y_s \frac{\mathrm{d}s}{r+s}$ converges in distribution or not. 
Remark, that $$X_t = Y_t + \int_0^t \frac{\mathrm{d}s}{r+s} Y_s = Y_t + \overline{\mu}_t.$$  Lemma \ref{tight} ensures that the duality for the weak convergence is true for functions bounded by $V$, so in particular for $f(x)=x$. Suppose that $V$ is such that $\underset{1\leq i\leq n}{\sum} a_i m_i =0$. Then, we know that $\frac{1}{t}\int_0^t Y_s \mathrm{d}s \xrightarrow {a.s.} 0$ and it remains to find the rate of convergence in order to conclude the proof. But, Bena\"im and Schreiber \cite{besch} (Theorem 1) have proved that, for an asymptotic pseudotrajectory (in probability) $Y$, the speed of convergence of the mean of the normalized occupation measure of $Y$ is the same as the speed of the pseudotrajectory. This means that the speed of convergence of the normalized occupation measure of the time-changed process $Y_{G^{-1}(t)}$ is $G^{-1}(1+t) - G^{-1}(t)$. Integrating by parts, we obtain 
\begin{eqnarray*}
\frac{1}{t} \int_0^t Y_s \mathrm{d}s 
=\frac{1}{t g(t)} \int_0^{G(t)} Y_{G^{-1}(u)} \mathrm{d}u + \frac{1}{t} \int_0^{G(t)} \mathrm{du}\frac{g'\circ
G^{-1}(u)}{(g\circ G^{-1}(u))^3} \int_0^u Y_{G^{-1}(s)} \mathrm{d}s.
\end{eqnarray*}
The first right-hand term converges to 0 because $G(t) \leq tg(t)$. It remains to prove the convergence of the second term. 
We have that (because, up to a multiplicative constant, $g\circ G^{-1}(u) = \log (2+u)$) 
\begin{eqnarray*}
\frac{1}{t} \int_0^{G(t)} \mathrm{du}\frac{g'\circ
G^{-1}(u)}{(g\circ G^{-1}(u))^3} s[G^{-1}(s+T) -G^{-1}(s)] \mathrm{d}s\\
\le \frac{G(t) (t-G^{-1}(G(t)+T)}{tg(t)} + \frac{1}{t} \int_0^t \frac{G^{-1}(s+T) -G^{-1}(s)}{g\circ G^{-1}(s)} \mathrm{d}s <\infty.
\end{eqnarray*}
So, if $V$ is a function such that $\underset{1\leq i\leq n}{\sum} a_i m_i \neq 0$, then $\int_0^t Y_s \mathrm{d}s$ does not converge. Suppose that $\underset{1\leq i\leq n}{\sum} a_i m_i = 0$. Let $U_t := \overline{\mu}_t$ and $V_t := X_t$. We conclude, because the celebrated Slutsky theorem asserts that for two sequences of $\mathbb{R}^d$-valued random variables $(U_t)$ and $(V_t)$,  with $U_t \underset{t\rightarrow \infty}{\xrightarrow {(d)}} U$ and
$|U_t-V_t| \underset{t\rightarrow \infty}{\xrightarrow {\mathbb{P}}} 0$, then $V_t \underset{t\rightarrow
\infty}{\xrightarrow {(d)}} U$.
\end{proof}
\begin{rk}
The condition on $g$ means that, asymptotically, $g(t) = k \log t$ for $k$ large enough.
\end{rk}

\section{The case $g(t)=1$}
We give the proof for $\R$ but it is easily reproduced in $\R^d$. We suppose now that $g(t)=1$. 
In order to study the behavior of the process $X$ solution of
$$ \mathrm{d}X_t = \mathrm{d}B_t-V'(X_t-\bar{\mu}_t) \mathrm{d}t,$$
we introduce the process $Y_t:=X_t-\bar{\mu}_t$, solution of the SDE
$$\mathrm{d}Y_t = \mathrm{d}B_t - V'(Y_t)\mathrm{d}t - \frac{Y_t}{r+t}\mathrm{d}t.$$ We also introduce the
Kolmogorov process $Z$ solution to $$\mathrm{d}Z_t = \mathrm{d}B_t -V'(Z_t) \mathrm{d}t.$$
This process is a positive recurrent diffusion. 
Denote by $\gamma$ its invariant probability measure, $\gamma(\mathrm{d}x) = \frac{e^{-2V(x)}}{\int_{\R} e^{-2V(y)}\mathrm{d}y} \mathrm{d}x$, and $\overline{\gamma}$ is the mean of $\gamma$. For all $h \in L^1(\gamma)$ we have, with an exponential speed of convergence, $$ \lim_{t\rightarrow\infty} \frac{1}{t} \int_0^t h(Z_s)\mathrm{d}s = \int_{\R} h\, \mathrm{d}\gamma \; \; a.s.$$
\begin{lemma}\label{lem:Yproba}
The process $Y$ converges in probability to a random variable
$Y_\infty$ of density $\gamma$ when $t$ goes to $+\infty$.
\end{lemma}
\begin{proof}
Using that $-y(y-z) \leq \frac{z^2}{2}$ for all $y,z \in \R$, we get
\begin{eqnarray*}
\frac{1}{2} \frac{\mathrm{d}}{\mathrm{d}t} \mathbb{E}(Y_t-Z_t)^2 &=& - \mathbb{E}(V'(Y_t)-V'(Z_t),Y_t-Z_t) - \mathbb{E}\left(\frac{Y_t}{r+t},Y_t-Z_t \right)\\
&\le & -(c+\tilde{C})\mathbb{E}(Y_t-Z_t)^2 +\frac{1}{r+t}\mathbb{E}(Z_t^2).
\end{eqnarray*}
Applying It\^o's formula, it is easy to prove the existence of $M>0$ such that for all $t\ge 0$, $\mathbb{E}(Z_t^2)\le M$. So, $\mathbb{E}(Y_t-Z_t)^2$ goes to zero. We choose the random variable $Z_0$, which distribution function is $\gamma$, so tha the law of $Z_t$ is $\gamma$ for all $t$. So, $Y$ converges in $L^2$ to a random variable of law $\gamma$.
\end{proof}

\begin{lemma}\label{lemmeamontrer2}
For all $A>0$, we have $\underset{t\rightarrow\infty}{\lim} \sup\limits_{0\leq s\leq A} \left|\int_0^s \left(Y_{e^{t+u}} - \overline{\gamma} \right) \mathrm{d}u\right| = 0$ a.s.
\end{lemma}
\begin{proof}
Fix $T>0$. We note $Y^T$ the process defined by $Y_s^T :=Y_{s+T}$. It is solution to 
$$\mathrm{d}Y_s^T = \mathrm{d}B_s^T - V'(Y_s^T) \mathrm{d}s - \frac{Y_s^T}{s+T} \mathrm{d}s.$$ Recall, that $\bar{\gamma} =\int_{\R} x\gamma(x) \mathrm{d}x$. Let $A>0$, we have for all $s\leq A$ and $T=e^t$:
\begin{eqnarray*}
\int_0^s \left(Y_{e^{t+u}} - \bar{\gamma} \right) \mathrm{d}u 
= \int_0^{T(e^s-1)} \frac{Y_{v+T}-Z_v^T}{v+T} \mathrm{d}v  + \int_0^{T(e^s-1)} \frac{Z_{v}^T-\bar{\gamma}}{v+T} =: I+J,
\end{eqnarray*}
where the process $Z^T$ is the solution to the SDE
\begin{equation*}
    \mathrm{d}Z_s^T = \mathrm{d}B_s^T - V'(Z_s^T) \mathrm{d}s, \, Z_0^T=Y_T.
\end{equation*}

1) Study of J. Integrating by part, we have with $S=e^s-1$,
\begin{eqnarray*}
J = \frac{TS}{TS+T} \left(\frac{1}{TS} \int_0^{TS} Z_u^T \mathrm{d}u - \bar{\gamma} \right) + \int_0^{TS} \frac{v}{(v+T)^2}\left(\frac{1}{v} \int_0^v Z_u^T \mathrm{d}u - \bar{\gamma} \right) \mathrm{d}v.
\end{eqnarray*}
The process $Z$ satisfies the limit-quotient theorem, so the first right-hand term converges a.s. to $0$ when $T$ goes to the infinity. For the second right-hand term we have
$$ \left| \int_0^{TS} \frac{v}{(v+T)^2} \left(\frac{1}{v} \int_0^v Z_u^T \mathrm{d}u - \bar{\gamma}\right) \mathrm{d}v\right|
\leq \frac{1}{T^2} \int_0^{TS} \left|\frac{1}{v} \int_0^v Z_u^T \mathrm{d}u - \bar{\gamma} \right| \mathrm{d}v.$$ By Ces\`aro, J converges to $0$ as $T$ goes to the infinity.

2) Study of I. We use the estimate of the distance between the processes $Y_v^T$ and $Z_v^T$ for $T$ large enough, as in Lemma \ref{lem:Yproba}. 
So, letting $C:=2(c-\tilde{C})>0$, we find the following upper bound
$$(Y_v^T-Z_v^T)^2 \leq e^{-Cv} \int_0^v \frac{e^{Cu}(Z_u^T)^2}{u+T} \mathrm{d}u.$$
Let $\sigma_\infty := \lim \frac{1}{t} \int_0^t (Z_u^T)^2 \mathrm{d}u$. Remind, that there exists $a>0$ such that the speed of convergence is less than $e^{-at}$, so that we have $$ \int_0^v
\frac{e^{Cu}(Z_u^T)^2}{u+T} \mathrm{d}u= \int_0^v
\frac{e^{Cu}\sigma_\infty}{u+T} \mathrm{d}u + \int_0^v \frac{e^{Cu}((Z_u^T)^2-\sigma_\infty)}{u+T} \mathrm{d}u =: K+L.$$ With probability 1, we obtain the upper bounds $|K| \leq \frac{e^{Cv}\sigma_\infty}{aT}$ and $|L| \leq \frac{e^{Cv}}{aT}$, implying $(Y_s^T - Z_s^T)^2 \leq \frac{\sigma_\infty +1}{aT}$ a.s. So, $|I|\le \frac{1}{\sqrt{T}}$ and $I$ converges to zero as $T$ goes to the infinity.
\end{proof}

We have know to study the asymptotic behavior of $\bar{\mu}_t$. Integrating by parts, we get
\begin{eqnarray}\label{formuleipp}
    \bar{\mu}_t = \bar{\mu}_0 + \int_0^t \frac{Y_s}{r+s}\mathrm{d}s =
    \bar{\mu}_0 + \frac{t}{r+t}m_t + \int_0^t \frac{s}{(r+s)^2}m_s \mathrm{d}s,
\end{eqnarray}
where $m_t:= \frac{1}{t}\int_0^t Y_s \mathrm{d}s$.

The proof of the main result of this section is based on the following:
\begin{propo}\label{lemmeamontrer}
The process $m_t$ converges in probability to $\overline{\gamma}$.
Moreover the speed of convergence of $m_t$ toward $\overline{\gamma}$ is less than $\frac{1}{t}$.
\end{propo}
\begin{proof}
As $\mathrm{d}m_t=\frac{1}{t}(Y_t-m_t)\mathrm{d}t$, letting $n_t:= m_{e^t}$, we get $\frac{\mathrm{d}}{\mathrm{d}t} n_t =Y_{e^t} - n_t$. Consequently, we find
\begin{eqnarray*}
n_{t+s}-n_t = \int_t^{t+s} (Y_{e^u} - n_u)\, \mathrm{d}u = \int_0^s (-n_{t+u} + \bar{\gamma})\, \mathrm{d}u + \int_0^s (Y_{e^{t+u}} - \bar{\gamma})\, \mathrm{d}u.
\end{eqnarray*}
Let $\varepsilon_t(s) := \int_0^s (Y_{e^{t+u}} - \bar{\gamma})\, \mathrm{d}u$. 
By Lemma \ref{lemmeamontrer2}, for all $A>0$ we have $\lim\limits_{t\to\infty}\sup\limits_{0\leq s\leq A} |\varepsilon_t(s)| = 0$ a.s., so $n_t$ is an asymptotic pseudotrajectory (a.s.) for the flow generated by $$\frac{\mathrm{d}\psi_t(x)}{\mathrm{d}t}= \bar{\gamma} - \psi_t(x), \; \; \psi_0(x)=x.$$ As this flow admits only one limit point which is exponentially attracted, the $\omega$-limit set of $n_t$ is reduced to $\{\bar{\gamma}\}$. So, $m_t=n_{\log t}$ converges a.s. to $m_\infty := \bar{\gamma}$.
\end{proof}

\begin{coro}
\begin{enumerate}
    \item If $\overline{\gamma}=0$, then $\bar{\mu}_t$ converges in probability as $t$ goes to infinity.
    \item If $\overline{\gamma}\neq 0$, then $\bar{\mu}_t$ diverges and $\lim \frac{\bar{\mu}_t}{\log t} = \overline{\gamma}$ (in probability).
\end{enumerate}
\end{coro}
\begin{proof}
By \eqref{formuleipp}, we get $$\bar{\mu}_t = \bar{\mu}_0 + \frac{t}{r+t}(m_t-\overline{\gamma}) + \int_0^t \frac{s}{(r+s)^2}(m_s-\overline{\gamma}) \mathrm{d}s + \overline{\gamma}\left(\log(1+t/r) -r\right).\qedhere$$
\end{proof}

Summarizing, we have now proved the following.
\begin{theo}
One of the following holds:
\begin{enumerate}
    \item If $\overline{\gamma} =0$, then $\bar{\mu}_t$ converges in probability to $\bar{\mu}_\infty$ and $X_t$ converges in probability to $Y_\infty +\bar{\mu}_\infty$, where the law of $Y_\infty$ has the density $\gamma$ ;
    \item Else, $X_t$ diverges.
\end{enumerate}
\end{theo}


\section{Appendix}
In Section \ref{s:asympY}, the optimal annealing schedule is not obtained. Actually, $k$ should be directly related to $m$, as Holley, Kusuoka and Stroock \cite{HKS} proved it for $a(t) \equiv \infty$. Let $z_0\in
\mathbb{R}^d$ such that $V_t(z_0) =0$. Let $K$ be the compact support of $\chi$.
\begin{defn}
The maximal height of the function $V_t$ is the non-negative
function $m(t)$ defined by
\begin{equation}
m(t) := \sup\{H_t(x,z_0) - V_t(x);\quad x\in K\},
\end{equation}
where $H_t(x,z) := \inf\{E_t(\gamma);\, \gamma \in
\mathcal{C}^1([0,1],K), \gamma(0)=x, \gamma(1)=z\}$ and

\noindent $E_t(\gamma) := \sup\{V_t(\gamma(u));\, u\in [0,1]\}.$
\end{defn}
Remark, that $m(t)$ does not depend on $z_0$ and so, we choose $z_0 =0$. The function $m(t)$ corresponds to the maximum of all the minimal energies one needs to go from each point of $\mathbb{R}^d$ to $z_0$. It is positive if and only if there exists several local minima.

\begin{lemma}
We have that $\underset{t\rightarrow \infty}{\lim} m(t) = m$, where $m$ is the maximal height function corresponding to $V$.
\end{lemma}
\begin{proof}
Let $M := \sup \{|x|^2; x\in K\}$. For all path $\gamma$, we have clearly $E_t(\gamma) \leq E(\gamma) +
\frac{M}{a(t)}$. Then, by definition we get $|H_t(x,0)-H_\infty(x,0)| \leq \frac{M}{a(t)}.$ Consequently, there exists $C>0$ such that $|m(t) - m(\infty)| \leq \frac{C}{a(t)}.$ 
\end{proof}
Jacquot relates in \cite{J} the height function to the second eigenvalue of the infinitesimal generator of
$Y^\varepsilon$ (that is the constant involved in the spectral gap inequality). He proves that $\underset{\varepsilon \rightarrow 0}{\lim} \varepsilon^2 \log \lambda_2(\infty,\varepsilon)= -2m(\infty)$. So, the ``critical'' value of $k$ should be $m$ instead of $\osc \chi$.

\bigskip

\noindent \textbf{S\'ebastien Chambeu}: Laboratoire de mod\'elisation
stochastique et statistique, Universit\'e Paris Sud, B\^atiment 425,
F-91405 Orsay Cedex, France.  Sebastien.Chambeu@math.u-psud.fr

\noindent \textbf{Aline Kurtzmann}: Institut Elie Cartan,
Universit\'e Henri Poincar\'e Nancy 1, B.P.239,  
F-54506 Vand\oe uvre-l\`es-Nancy Cedex, France. Aline.Kurtzmann@iecn.u-nancy.fr
\end{document}